\RequirePackage{fix-cm}

\documentclass{article}     

\usepackage[T1]{fontenc}       
\usepackage{lmodern}           
\usepackage{amsmath}           
\usepackage{amsfonts}
\usepackage{graphicx}
\usepackage{complexity}%
\usepackage{enumitem}
\usepackage{fullpage}%
\usepackage{hyperref}
\usepackage{lineno}
\usepackage[framemethod=tikz]{mdframed}
\usepackage{pict2e}
\usetikzlibrary{arrows.meta}

\usepackage{subcaption}

\usepackage{amsthm}
\newtheorem{theorem}{Theorem}
\newtheorem{proposition}{Proposition}
\newtheorem{corollary}{Corollary}
\newtheorem{definition}{Definition}
\newtheorem{lemma}{Lemma}
\usepackage[affil-it]{authblk}
\newtheorem{observation}[theorem]{Observation}

\DeclareMathOperator{\pthree}{P_3}
\DeclareMathOperator{\pthrees}{P_3^*}

\DeclareMathOperator{\opthree}{\pthree}
\DeclareMathOperator{\opthrees}{\pthrees}

\DeclareMathOperator{\hull}{hull}

\DeclareMathOperator{\hullops}{\hull_{\opthrees}}

\newcommand{\phullop}[1]{\hull_{\scriptscriptstyle\opthree}^{#1}}
\newcommand{\phullops}[1]{\hull_{\scriptscriptstyle\opthrees}^{#1}}

\DeclareMathOperator{\pfour}{P_4}
\DeclareMathOperator{\ofour}{O_4}

\DeclareMathOperator{\Path}{P}
\DeclareMathOperator{\CapO}{O}
\DeclareMathOperator{\Cycle}{C}
\newcommand{\pn}[1]{\ensuremath{\overrightarrow{\Path_{#1}}}}
\newcommand{\on}[1]{\ensuremath{\overrightarrow{\CapO_{#1}}}}
\newcommand{\cn}[1]{\ensuremath{\overrightarrow{\Cycle_{#1}}}}

\DeclareMathOperator{\opfour}{\overrightarrow{\pfour}}
\DeclareMathOperator{\oofour}{\overrightarrow{\ofour}}

\DeclareMathOperator{\ext}{Ext}

\DeclareMathOperator{\extps}{\ext_{\scriptscriptstyle{\opthrees}}}
\newcommand{\pextp}[1]{\ext_{\scriptscriptstyle{\opthree}}^{#1}}
\newcommand{\pextps}[1]{\ext_{\scriptscriptstyle{\opthrees}}^{#1}}

\DeclareMathOperator{\intfunc}{I}
\DeclareMathOperator{\intfuncops}{\intfunc_{\opthrees}}
\newcommand{\pintfuncops}[1]{\intfunc_{\opthrees}^{#1}}

\DeclareMathOperator{\dist}{dist}
\newcommand{\odist}{\ensuremath{\overrightarrow{\dist}}}

\begin{document}
%\linenumbers
\title{Convex geometries and directed paths on three vertices}

\author[1]{J\'ulio Ara\'ujo%
  %\thanks{Electronic address: \texttt{julio@mat.ufc.br}}
  }
\affil[1]{Departamento de Matem\'atica, Universidade Federal do Cear\'a, Brazil}

\author[2]{Mitre C. Dourado%
  %\thanks{Electronic address: \texttt{mitre@ic.ufrj.br}}
  }
\affil[2]{Instituto de Computa\c{c}\~{a}o, Universidade Federal do Rio de Janeiro,  Brazil}

\author[3]{Marisa Gutierrez%
  %\thanks{Electronic address: \texttt{marisa@mate.unlp.edu.ar}}
  }
\affil[3]{Facultad de Ciencias Exactas, Universidad Nacional de La Plata, Argentina}

\author[4]{Fabio Protti%
  %\thanks{Electronic address: \texttt{fabio@ic.uff.br}}
  }
\affil[4]{Instituto de Computa\c{c}\~{a}o, Universidade Federal Fluminense, Brazil}

\author[3]{Silvia Tondato%
  %\thanks{Electronic address: \texttt{tondato@mate.unlp.edu.ar}}
  }

\affil[ ]{\small{julio@mat.ufc.br, mitre@ic.ufrj.br, \{marisa,tondato\}@mate.unlp.edu.ar,fabio@ic.uff.br}}

% \author{J\'ulio Ara\'ujo \and Mitre C. Dourado\and Marisa Gutierrez \and Fabio Protti \and Silvia Tondato}

% \institute{J\'ulio Ara\'ujo at Departamento de Matem\'atica, Universidade Federal do Cear\'a, Brazil. \email{julio@mat.ufc.br}\\ \and 
% Mitre C. Dourado at Instituto de Computa\c{c}\~{a}o, Universidade Federal do Rio de Janeiro,  Brazil. \email{mitre@ic.ufrj.br}\\ \and Fabio Protti at
% Instituto de Computa\c{c}\~{a}o, Universidade Federal Fluminense, Brazil. \email{fabio@ic.uff.br} \\ \and  Marisa Gutierrez \and Silvia Tondato at CMaLP, Facultad de Ciencias Exactas, Universidad Nacional de La Plata, Argentina. \email{marisa@mate.unlp.edu.ar, tondato@mate.unlp.edu.ar}}

%%%%%%%%%%%%%%%%%%%%%%%%%%%%%%%%%%%%%%%%%%%%
\date{June 23, 2026}

\maketitle

\begin{abstract} 

A \emph{convexity space} is an ordered pair $(V,\mathcal{C})$, where $V$ is an arbitrary set and $\mathcal{C}$ is a family of subsets of $V$, called \emph{convex sets}, which contains $\{\emptyset,V\}$ and is closed under intersections and nested unions of its elements. For any $S\subseteq V$, the \emph{convex hull} of $S$ is the inclusion-wise minimum convex set $C\in \mathcal{C}$ such that $S\subseteq C$. For a convex set $C\in \mathcal{C}$, an element $p\in C$ is an \emph{extreme} of $C$ if $p$ does not belong to the convex hull of $C\setminus\{p\}$. A convexity $\mathcal{C}$ defined over $V$ is a \emph{convex geometry} (a.k.a. \emph{geometric convexity}) if any convex set $C\in \mathcal{C}$ is the convex hull of its extreme elements.

Given an oriented graph $D = (V,A)$, the family $\mathcal{C}$ of subsets of $V$ is the $\opthree$-convexity defined over $D$ if $\mathcal{C}$ is formed by all (convex) sets $C\subseteq V$ satisfying the following property: no vertex $v\in V\setminus C$ is the central vertex of a directed path $P=(u,v,w)$ with $\{u,w\} \subseteq C$. In the
$\opthrees$-convexity defined over $D$, we have that $\mathcal{C}$ is formed by all (convex) sets $C\subseteq V$ satisfying the following property: no vertex $v\in V\setminus C$ is the central vertex of a directed path $P=(u,v,w)$ such that $\{u,w\} \subseteq C$ and $(u,w)\notin A$.

In this work, we present necessary and sufficient conditions over an oriented graph $D$ so that the $\opthree$-convexity over $D$ is geometric, or the $\opthrees$-convexity over $D$ is geometric. While the first case implies a polynomial-time algorithm to decide whether the $\opthree$-convexity over $D$ is a geometric, we show that it is $\coNP$ to decide whether the $\opthrees$-convexity over $D$ is a convex geometry. We also present a family of oriented graphs, termed acyclic indifference oriented graph, and demonstrate that deciding whether the $\opthrees$-convexity over an acyclic indifference oriented graph is geometric can be solved in polynomial-time.

\end{abstract}

%%%%%%%%%%%%%%%%%%%%%%%%%%%%%%%%%%%%%%%%%%%%%%%%%%
\section{Introduction} \label{sec:intro}

For further notions on Graph Theory and Abstract Convexity, the reader is referred to~\cite{BM2008,vandeVel1993}.
For an extensive
overview of directed graphs, see~\cite{BG2008}. In this work, we only consider simple and finite (directed) graphs.

A \emph{convexity space} is an ordered pair $(V,\mathcal{C})$, where $V$ is an arbitrary set, and $\mathcal{C}$ is a family of subsets of $V$, called \emph{convex sets}, that satisfies:
\begin{enumerate}[label=\alph*)]
  \item $\emptyset,V\in \mathcal{C}$;
  \item For all $\mathcal{C'}\subseteq \mathcal{C}$, we have $\bigcap \mathcal{C'}\in \mathcal{C}$;
  \item Every nested union of convex sets is a convex set.
\end{enumerate}

In the case of a finite set $V$, the condition of nested unions is always satisfied, and we can only consider the first two.

We often say that the family $\mathcal{C}$ is a \emph{convexity} over $V$. For any $S\subseteq V$, the \emph{convex hull} of $S$ is the inclusion-wise minimum convex set $C\in \mathcal{C}$ such that $S\subseteq C$. For a convex set $C\in \mathcal{C}$, an element $p\in C$ is an \emph{extreme} of $C$ if $p$ does not belong to the convex hull of $C\setminus\{p\}$.

For a positive integer $d$, it is well-known that a convex set in $\mathbb{R}^d$ is any subset $S\subseteq \mathcal{R}^d$ such that, for any two points $p_1,p_2\in S$, the points in the straight-line segment linking $p_1$ to $p_2$ must also belong to $S$. In this context, a well-known result is that any convex set in $\mathbb{R}^d$ that is closed and bounded corresponds to the convex hull of its extreme points~\cite{KreinMilman1940,Minkowski1911}. More generally, a convexity $\mathcal{C}$ defined over $V$ is a \emph{convex geometry} (a.k.a. \emph{geometric convexity}) if any convex set $C\in \mathcal{C}$ is the convex hull of its extreme elements. When it is not necessarily the case, given a convex set $C$, we say that $C$ is {\em geometric} if $C$ is the convex hull of its extreme elements.

When the set $V$ is the set of vertices of a (directed) graph, the convexity space $(V,\mathcal{C})$ is a \emph{graph convexity}. There are several graph convexities defined in the literature~\cite{ADPS2025}. The study of graph convexities has received a lot of attention from the community, especially in the last 25 years. Most of them study (non-directed) graphs, even if some of the older works in this domain address the oriented case~\cite{erdHos1972some,pfaltz1971convexity}.  

Most graph convexities are defined in terms of (directed) paths of the input (resp. directed) graph. More precisely, the convex sets $C$ in the family $\mathcal{C}$ satisfy the following property: if $u,w\in C$, then all internal vertices of (resp. directed) $(u,w)$-paths satisfying some constraint $\Pi$ must also belong to $C$. When the constraint $\Pi$ is that the path is a shortest path, then the graph convexity is called \emph{geodesic}. If $\Pi$ corresponds to the property of being an induced path, then the convexity is called \emph{monophonic}. When $\Pi$ requires that the path has length two, we have the \emph{$\opthree$}-\emph{convexity}, while if the path is further required to be an \emph{induced/shortest} path on three vertices, then such convexity is known as \emph{$\opthrees$}-\emph{convexity}. 

Note that, by definition, for any $u,w \in C$, one must consider all directed $(u,w)$-paths and all directed $(w,u)$-paths satisfying the constraint $\Pi$, which of course, may not necessarily contain the same internal vertices in the directed case.

Graph convexities defined over directed graphs in the literature usually assume that they are indeed oriented graphs, which is a hypothesis that we also consider in this work. Recall that $D$ is an \emph{oriented graph} if it is an orientation of a simple graph $G$, i.e. $D$ is obtained from $G$ by replacing each edge of $G$ by exactly one arc with the same endpoints. Consequently, $D$ is a directed graph without loops, nor parallel, or symmetric arcs. 

Although graph convexity has been more extensively studied than digraph convexity, there is a substantial body of literature dedicated to $\opthree$-convexity \cite{AMM+2026,erdHos1972some,PWW.08}.

In a recent work~\cite{AMM+2026}, the authors define the $\opthrees$-convexity in an oriented graph $D$ in such a way that $S$ is a convex set if for every $u,w\in S$, we have that any vertex $v$ lying in a shortest directed $(u,w)$-path or $(w,u)$-path of length two must also belong to $S$. They focus on the computational complexity of determining some graph convexity parameters in this convexity.

Our goal in this work is to study which are the oriented graphs $D$ such that the $\opthrees$-convexity defined over $V(D)$ is geometric. There are several similar studies for classes of non-oriented graphs. For a recent survey on the topic, see~\cite{DGPST2025}.

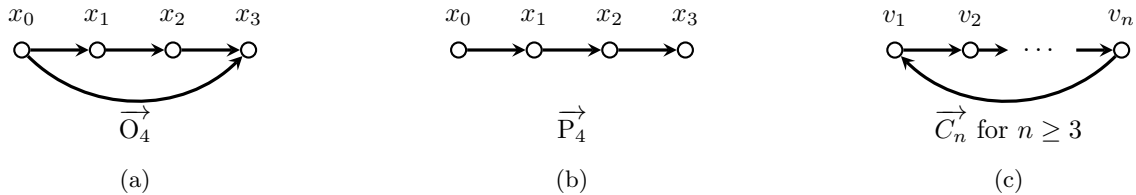
\begin{figure}[!ht]
    \centering
    \begin{subfigure}[b]{0.3\textwidth}
    \centering
    \begin{tikzpicture}
        \node (ofour) at (0,-1) {$\oofour$};
    
      \begin{scope}[every node/.style={draw, fill=white, circle, thick, fill = white, inner sep=2pt}]
        \node (x0) at (-1.5,0) [label=above:$x_0$]{};
        \node (x1) at (-0.5,0) [label=above:$x_1$]{};
        \node (x2) at (0.5,0) [label=above:$x_2$]{};
        \node (x3) at (1.5,0) [label=above:$x_3$]{};
      \end{scope}
      
      \begin{scope}[>={stealth[black]},
        % 			  every node/.style={fill=white,circle},
        every edge/.style={draw=black,very thick}]
        \path [->] (x0) edge node  {} (x1);
        \path [->] (x1) edge node  {} (x2);
        \path [->] (x2) edge node  {} (x3);
        \path [->] (x0) edge [bend right=45] node  {} (x3);
      \end{scope}
    \end{tikzpicture}    
    \caption{}
    \label{fig:ofour}
    \end{subfigure}\hfill
    \begin{subfigure}[b]{0.3\textwidth}
    \centering
    \begin{tikzpicture}
        \node (ofour) at (0,-1) {$\opfour$};
    
      \begin{scope}[every node/.style={draw, fill=white, circle, thick, fill = white, inner sep=2pt}]
        \node (x0) at (-1.5,0) [label=above:$x_0$]{};
        \node (x1) at (-0.5,0) [label=above:$x_1$]{};
        \node (x2) at (0.5,0) [label=above:$x_2$]{};
        \node (x3) at (1.5,0) [label=above:$x_3$]{};
      \end{scope}
      
      \begin{scope}[>={stealth[black]},
        % 			  every node/.style={fill=white,circle},
        every edge/.style={draw=black,very thick}]
        \path [->] (x0) edge node  {} (x1);
        \path [->] (x1) edge node  {} (x2);
        \path [->] (x2) edge node  {} (x3);
      \end{scope}
    \end{tikzpicture}
    \caption{}
    \label{fig:pfour}
    \end{subfigure}\hfill
    \begin{subfigure}[b]{0.3\textwidth}
    \centering
   \begin{tikzpicture}
        \node (cn) at (0,-1) {$\overrightarrow{C_n}$ for $n\ge 3$};
        \node (ret) at (0.4,0) {$\ldots$};
    
      \begin{scope}[every node/.style={draw, fill=white, circle, thick, fill = white, inner sep=2pt}]
        \node (x0) at (-1.5,0) [label=above:$v_1$]{};
        \node (x1) at (-0.5,0) [label=above:$v_2$]{};
        
        \node (x3) at (1.5,0) [label=above:$v_n$]{};
      \end{scope}
      
      \begin{scope}[>={stealth[black]},
        % 			  every node/.style={fill=white,circle},
        every edge/.style={draw=black,very thick}]
        \path [->] (x0) edge node  {} (x1);
        \path [->] (x1) edge node  {} (0,0);
        \path [->] (0.9,0) edge node  {} (x3);
        \path [->] (x3) edge [bend left=45] node  {} (x0);
      \end{scope}
    \end{tikzpicture}
    \caption{}
    \label{fig:cn}
    \end{subfigure}
    \caption[]{Some forbidden induced subdigraphs for $\opthrees$-geometries.}
    \label{fig:forbiddenInducedSubgraphs}
  \end{figure}

In this work, we find necessary and sufficient conditions on an oriented graph $D$ 
 such that the $\opthree$-convexity over $D$ is geometric. Subsequently, our work focused on $\opthrees$-convexity. We first present necessary and sufficient conditions on an oriented graph $D$ such that the $\opthrees$-convexity over $D$ is geometric. Namely, we define a family of oriented graphs $\mathcal{H}$ (see Definition \ref{def:familyH}) such that the $\opthrees$-convexity defined over the vertex set of an oriented graph $D$ is geometric if and only if $D$ is a $\{\opfour, \oofour\}$-free directed acyclic graph (DAG) such that for every set $S \subseteq V(D)$, if $D[S] \in \mathcal{H}$, then the convex hull of $S$ is geometric. Then we show that deciding whether a given oriented graph $D$ has a geometric $\opthrees$-convexity is a $\coNP$-complete problem. For last, we
   present a new class of oriented graphs in which deciding if an oriented graph is a geometry with respect
to $\opthrees$-convexity is polynomial.

Section \ref{sec:prelim} contains the necessary definitions and notation used throughout the work. Section \ref{sec:P3-conv}
  provides the characterization of the $\opthree$-geometries. In Section \ref{sec:familyH}, we present the family $\mathcal{H}$ and the characterization of $\opthrees$-geometries. In Section \ref{sec:coNP}, we show our computational complexity hardness result. 
In Section \ref{sec:acycleindif},
  we present the class of acyclic indifference oriented graphs, and we demonstrate that this class is a convex geometry under $\opthrees$-convexity, which is polynomial-time solvable.

Finally, in Section \ref{sec:further}, we present avenues for further research.

\section{Preliminaries}
\label{sec:prelim}

Given a directed $u,v$-path $P=(u=v_1,a_1,\ldots, a_{k-1},v_{k}=v)$ in a digraph $D$, we say that $P$ is \emph{induced} if, for any $i,j\in\{1,\ldots, k\}$ such that $i<j$ and $j\neq i+1$, $(v_i,v_j)\notin A(D)$. Note that, in principle, arcs from $v_j$ to $v_i$ may occur for $j>i$, although in most cases in this work we will treat DAGs. Observe also that a shortest directed $(u,w)$-path of length two is the same as an induced directed path of length two. For every $x,y\in V(D)$, let $\odist_D(x,y)$ denote the minimum length of a directed $(x,y)$-path, when it exists, i.e., the \emph{distance} from $x$ to $y$. It is considered to be $+\infty$ when there is no such directed path. We omit the subscript when $D$ is clear in the context.

Given a digraph $D$ and $S\subseteq V(D)$, we denote by $D[S] = D-(V(D)\setminus S)$ the subdigraph of $D$ induced by $S$. We say that $D$ is \emph{$H$-free}, for some digraph $H$, if $D$ has no induced subdigraph isomorphic to $H$. For a family of graphs $\mathcal{F}$, $D$ is \emph{$\mathcal{F}$-free} if $G$ is $H$-free for every $H\in \mathcal{F}$.

Given an oriented graph $D$, if $(u, v)\in A(D)$, then $u$ is an \emph{in-neighbor} of $v$, while $v$ is an \emph{out-neighbor} of $u$. A vertex $v$ is a \emph{source} in $D$ if it has no in-neighbors in $D$; it is a \emph{sink} in $D$ if it has no out-neighbors in $D$. A vertex $v$ is \emph{transitive} in $D$ if, for any in-neighbor $u$ of $v$ and any out-neighbor $w$ of $v$, we have that $(u,w)\in A(D)$.

\paragraph{Graph Convexity.}  Let $G$ be a graph. Some of the most well-known graph convexities are those defined by interval functions via systems $\cal{P}$ of paths in $G$ \cite{DGPST2025}. 
A \emph{convex} subset $S\subseteq V(G)$ in the $\mathcal{P}$-convexity defined over $V(G)$ is then a fixed point of such an interval function, i.e., $I_{\mathcal{P}}(S) = S$. It is well-known that the \emph{convex hull} of a given subset $S\subseteq V(G)$, denoted here by $hull_{\mathcal{P}}^{G}(S)$, can be obtained by successive compositions of the interval function with itself. If $S$ is a convex set, recall that $v\in S$ is an \emph{extreme} vertex of $S$ if $v\notin hull_{\mathcal{P}}^{G}(S\setminus\{v\})$. 
We denote the set of extreme vertices of the convex set $S\subseteq V(G)$ in the $\mathcal{P}$-convexity by $Ext_{\mathcal{P}}^{G}(S)$. We omit the superscript $G$ when the graph $G$ is clear in the context.
A graph $G$ is called a \emph{$\mathcal{P}$-geometry} (or a \emph{convex geometry}, or even a \emph{geometric convexity}) in the $\mathcal{P}$-convexity if the family of convex subsets of $V(G)$ in the $\mathcal{P}$-convexity defines a geometric convexity over $V(G)$.

\paragraph{$\opthree$-Convexity.}
Let $D$ be an oriented graph. For every $S\subseteq V(D)$, define the \emph{interval} of $S$ in the $\opthree$-convexity as $I_{{P}_3}(S) = S\cup \{v\mid \text{there exist } u,w\in S \text{ such that $v$ lies in any } \text{directed } (u,w)\text{-path of length two}\}$

\begin{proposition}[\cite{AMM+2026}]
  \label{prop:extremep3}
  Let $D$ be an oriented graph. Then, a vertex $x$ is an extreme vertex of $D$ in the $\opthree$-convexity if and only if $x$ is a source or sink vertex in $D$.  
\end{proposition}

\paragraph{$\opthrees$-Convexity.} Let $D$ be an oriented graph. For every $S\subseteq V(D)$, define the \emph{interval} of $S$ in the $\opthrees$-convexity as 
$\intfuncops(S) =S\cup \{v\mid$ there exist $u,w\in S$ such that $v$ lies in any shortest directed $(u,w)$-path of length two$\}$. Note that there is no arc $(u,w)$.

\begin{proposition}[\cite{AMM+2026}]
\label{prop:extremep3s}
Let $D$ be an oriented graph. Then, a vertex $x$ is an extreme vertex of $D$ in the $\opthrees$-convexity if and only if $x$ is a source,  sink, or transitive vertex in $D$.
\end{proposition}

An argument that we repeatedly use in this work is the following:

\begin{observation}
\label{obs:added-non-extreme}
    Given an oriented graph $D$ and a non-empty set of vertices $S\subseteq V(D)$, then any vertex $v\in \hullops(S)\setminus S$ is not an extreme vertex of $\hullops(S)$.
\end{observation}

Indeed, note that if $v\in  \hullops(S)\setminus S$, if one thinks of the $\hullops(S)$ being computed by successively composing the interval function $\intfuncops(S)$, it means that $v$ is added to $\hullops(S)$ thanks to the existence of a shortest directed $(u,w)$-path of length two whose internal vertex is $v$, for some vertices $u,w$ added to $\hullops(S)$ before $v$. Consequently, $v$ cannot be extreme by Proposition \ref{prop:extremep3s}. 

Another useful argument deals with subdigraphs and convex sets.

\begin{observation}
\label{obs:deleteArc}
Let $S$ be a convex set, in the $\opthrees$-convexity, of an oriented graph $D$. If the arc $(u,v) \in A(D)$ has at most one vertex of $S$, then $S$ is convex in $D-(u,v)$. 
\end{observation}

Let $\on{n}$ be the oriented graph whose set of vertices is
$\{x_i\mid i\in\{0,\ldots, n-1\}\}$ and the set of arcs is $\{(x_i,x_{i+1})\mid i\in\{0,\ldots, n-2\}\cup \{(x_0,x_{n-1})\}$. We say that $x_0$ and $x_{n-1}$ are the \emph{endpoints} of $\on{n}$. We denote by $\cn{n}$ the directed cycle on $n$ vertices and by $\pn{n}$ the directed path on $n$ vertices. See Figure \ref{fig:forbiddenInducedSubgraphs}.

\begin{proposition} \label{pro:small} 
The oriented graphs $\pn{n}$ and $\on{n}$, for every $n\geq 4$, and the oriented graphs $\cn{n}$, for any integer $n\geq 3$, are not $\opthrees$-geometry.
\end{proposition}
\begin{proof}
    Note that $V(\cn{n})$ is a convex non-empty set of $\cn{n}$, since $n\geq 3$, but it has no extreme points as presented in Proposition \ref{prop:extremep3s}, and thus it cannot be the convex hull of its extreme points. With respect to $\on{n}$ and $\pn{n}$, note that in both cases their endpoints $x_0$ and $x_{n-1}$ are the extreme vertices of the convex set composed by all vertices, but $\hullops(\{x_0,x_{n-1}\}) = \{x_0,x_{n-1}\}$.
\end{proof}

%%%%%%%%%%%%%%%%%%%%%%%%%%%%%%%
% New P3 convexity section
%%%%%%%%%%%%%%%%%%%%%%%%%%%%%%%

\section{\texorpdfstring{$\opthree$}{P3}-convexity}
\label{sec:P3-conv}

We start our main contribution by presenting a characterization of convex geometries in the $\opthree$-convexity that naturally leads to a polynomial-time algorithm to recognize such graphs.

\begin{theorem}\label{thm:P3-convexity}
An oriented graph $D$ is a convex geometry in the $\opthree$-convexity if and only if all the following conditions are valid:
\begin{enumerate}
\item\label{item:acyclic} $D$ is acyclic;
\item\label{item:dist2} for all distinct $x,y \in V(D)$ such that $y$ is a descendant of $x$, we have $\odist(x,y)\leq 2$; and
\item\label{item:noinducedO4} if $D$ has an induced subdigraph isomorphic to an $\oofour$ with vertex set $\{a,b,c,d\}$, being $a$ and $d$ its endpoints, then $\{b,c\}\subseteq \phullop{D}(\{a,d\})$.
\end{enumerate}
\end{theorem}

\begin{proof}
    Assume first that $D$ is a convex geometry in the $\opthree$-convexity. If $D$ contains an oriented cycle $C$, then $\phullop{D}(V(C))$ is a convex set with no extreme vertices, thanks to Proposition \ref{prop:extremep3} and to Observation \ref{obs:added-non-extreme}, a contradiction. Thus, Condition~\ref{item:acyclic} is valid.

    Suppose now that $\odist(x,y)=k>2$, for some $x,y\in V(D)$ such that $y$ is a descendant of $x$. Let $P$ be a directed path from $x$ to $y$ of length $k$ in $D$. Since $D$ is acyclic and $P$ is a shortest directed $(x,y)$-path, note that $P$ is an induced subdigraph of $D$ isomorphic to $\pn{k}$, $k\geq 4$. Then $X=\phullop{D}(V(P))$ is a convex set such that $\{x,y\} \subseteq \ext(X)$, but $X \neq \phullop{D}(\{x,y\})=\{x,y\}$, a contradiction. Thus, Condition~\ref{item:dist2} is valid.

    Finally, suppose that $R=(a,b,c,d)$ is a directed path such that the subdigraph $H$ of $D$ induced by $V(R)$ is isomorphic to $\opfour$ or $\oofour$. Let $Y=\phullop{D}(V(R))$. Note that $Y$ is a convex set, $\{a,b\} = \ext(Y)$ (recall Observation \ref{obs:added-non-extreme}), and $\{b,c\}\subseteq Y$. Since $D$ is a convex geometry, we have that $Y=\phullop{D}(\{a,d\})$. Thus, Condition~\ref{item:noinducedO4} is valid.
    
     Conversely, suppose that Conditions \ref{item:acyclic}, \ref{item:dist2}, and \ref{item:noinducedO4} are valid and $D$ is not a convex geometry, that is, $D$ has an induced subdigraph $G$ such that $V(G)$ is a convex set in the $\opthree$-convexity defined over $D$ which is not geometric. Then, there exists $x\in V(G)$ such that $x\not\in Y=\phullop{G}(\pextp{G}(V(G))=\phullop{D}(\pextp{G}(V(G))$. 

    By Condition \ref{item:acyclic}, note that $G$ is acyclic. Observe that $x$ is not an extreme vertex of $G$ and thus $x$ is not a source nor a sink vertex of $G$, thanks to Proposition \ref{prop:extremep3}. Thus, let $R$ be a maximal directed path in $G$ containing $x$ as one of its inner vertices. It is clear that $R$ starts at a source $s$ and ends at a sink $t$ of $G$, since $G$ is acyclic. Also, observe that $s,t \in \pextp{G}(V(G))$. As $s\in Y$ and $x\notin Y$, let $a$ and $b$ be the first pair of consecutive vertices of $R$ such that $a\in Y$ and $b\not\in Y$. Let $d=t$. If $\odist(b,d)=1$, then $b$ would belong to $Y$, a contradiction. Thus, by Condition~\ref{item:dist2}, $\odist(b,d)=2$, and this implies that there exists $c\in V(G)$ such that $(b,c,d)$ is an induced directed path. Since $b\not\in Y$,  it follows that $c\not\in Y$. Also, $\odist(a,c)=2$, otherwise $c$ would belong to $Y$. We conclude that  $R'=(a,b,c,d)$ is a directed path such that the subdigraph of $G$ induced by $V(R')$ is isomorphic to $\opfour$ or $\oofour$. However, Condition~\ref{item:noinducedO4} tells us that $\{b,c\}\subseteq \phullop{D}(\{a,d\})\subseteq Y$, a contradition. Therefore, $V(G)$ is geometric and $D$ is a convex geometry.
    
\end{proof}

Since Conditions \ref{item:acyclic}, \ref{item:dist2}, and \ref{item:noinducedO4} can be easily checked in polynomial time, the corollary below is straightforward.

\begin{corollary}
  \label{cor:p3inP}
Let $D$ be an oriented graph. Then, deciding whether $D$ is a convex geometry in the $\opthree$-convexity can be done in polynomial time.  
\end{corollary}

Of course, a brute-force algorithm implicitly suggested by Corollary \ref{cor:p3inP} would take, in a worst case scenario, $\mathcal{O}(n^7)$-time to try all subsets of four vertices to look for an induced $\oofour$. For each one, run a cubic algorithm to compute the convex hull of its endpoints and verify the conditions. A natural question is whether one can do better than this brute-force algorithm. We leave this as an open problem in Section \ref{sec:further}.

\section[convexity]{$\opthrees$-convexity}

In this section, we focus on the contributions for $\pthrees$-convexity. We first present in Section \ref{sec:familyH} a characterization of oriented graphs that define a convex geometry; however, in this case it will not provide a polynomial-time recognition algorithm for the graphs in this class.

In fact, in Section \ref{sec:coNP} we show that it is \coNP-complete to determine whether a given oriented graph $D$ defines a geometric convexity in the $\opthrees$-convexity.

Finally, in Section \ref{sec:acycleindif}, we present a particular family that allows polynomial-time recognition.

\subsection{Characterizing the oriented geometric graphs} \label{sec:familyH}

An oriented graph $D$ is \emph{$H$-free}, given the oriented graph $H$, if $H$ is not an induced subdigraph of $D$. For a family of oriented graphs $\mathcal{F}$, $D$ is \emph{$\mathcal{F}$-free} if $D$ is $H$-free for every $H\in \mathcal{F}$. Recall that DAG is an abbreviation for \emph{directed acyclic graph}.

In the sequel, we define the class of oriented graphs that we prove in this section to be the class of $\opthrees$-geometries.

Let $u,v$ be vertices of a DAG $D$. If there is a $(u,v)$-path in $D$, then we say that $u$ is an \emph{ancestor} of $v$ and that $v$ is a {\em descendant} of $u$ in $D$.

\begin{definition}
	\label{def:familyH}
	We say that $H$ belongs to class $\mathcal{H}$ if $H$ is a vertex minimal DAG that is $\{\opfour, \oofour\}$-free such that $V(H)$ can be partitioned into three sets $B, C_{\ell}$ and $C_r$ such that there is a circular ordering $c^1, \ldots, c^{\ell'+r'}$, where $\ell' = |C_{\ell}|$ and $r' = |C_r|$, of the vertices of $C_{\ell} \cup C_r$ satisfying:

\begin{enumerate}
	\item $\min\{|B|, \ell', r'\} \ge 2$; \label{ite:t}

	\item $|B| \le \ell' + r'$; \label{ite:non-empty}	
	
	\item no vertex of $C_{\ell} \cup C_r$ is the internal vertex of an induced path with both endpoints in $B$; \label{ite:Cconcave}

   	\item for every positive $i \le \ell' + r'$, if $c^i \in C_\ell$, then there is $b \in B$ such that $(b,c^i,c^{i+1})$ is an induced path; and \label{ite:Clnoextreme}

	\item for every positive $i \le \ell' + r'$, if $c^i \in C_r$, then there is $b \in B$ such that $(c^{i+1},c^i,b)$ is an induced path. \label{ite:Crnoextreme}
\end{enumerate}
\end{definition}
In Definition \ref{def:familyH}, we emphasize that $c^{i+1} = c^1$ when $i = \ell'+r'$. This is what we mean by circular ordering of $c^1,\ldots, c^{\ell'+r'}$.

\begin{lemma} \label{lem:cons}
Consider an $H \in \mathcal{H}$ with partition $(B, C_{\ell}, C_r)$ of $V(H)$ satisfying Definition \ref{def:familyH}. Then
	
	\begin{enumerate}[label=\roman*)]

		\item For every positive $i \le \ell' + r'$, if $c^i \in C_\ell$, then $(c^i,c^{i+1}) \in A(H)$; otherwise, $(c^{i+1},c^i) \in A(H)$. \label{ite:lrcycle}

		\item No vertex of $C_{\ell} \cup C_r$ is extreme. \label{ite:Cnoextreme}
		
		\item For any $c \in C_\ell \cup C_r$, it holds $V(H) = \phullops{H}(B \cup \{c\})$. \label{ite:one}	
	\end{enumerate}	
	
\end{lemma}

\begin{proof}
	\noindent Item \ref{ite:lrcycle} and Item \ref{ite:Cnoextreme} follow directly from the conditions in Definition \ref{def:familyH}, namely Item \ref{ite:Clnoextreme} and Item \ref{ite:Crnoextreme}.
	To prove Item \ref{ite:one}, it suffices to show that for any $i \in [\ell' + r']$, the predecessor of $c^i$ in the given circular ordering belongs to $\phullops{H}(B \cup \{c^i\})$.	
	First, consider that $c^{i-1} \in C_\ell$.
	Since $(b,c^{i-1}, c^i)$ is an induced path, for some $b \in B$, we have that $c^{i-1} \in \phullops{H}(B \cup \{c^i\})$.
	Now, consider that $c^{i-1} \in C_r$.
Since $(c^i,c^{i-1},b')$ is an induced path for some $b' \in B$, we have that $c^{i-1} \in \phullops{H}(B \cup \{c^i\})$.
\end{proof}

\begin{proposition}
	Every oriented graph of $\mathcal{H}$ is not a convex geometry in the $\opthrees$-convexity.    
\end{proposition}

\begin{proof}
	Let $H$ be any member of $\mathcal{H}$. It suffices to present a subset $S \subseteq V(H)$ such that $S$ is convex, but it is not the convex hull of its extreme vertices. Take $S = V(H)$.
	By definition, $S$ is a convex set. 
	By Item \ref{ite:t}, Item \ref{ite:Clnoextreme} and Item \ref{ite:Crnoextreme} of Definition \ref{def:familyH}, we know that $C_r\cup C_{\ell}$ is non-empty and composed of vertices that are not extreme.
	By Item \ref{ite:Cconcave}, we have that $\phullops{H}(B) = B$, which means that $S$ is not the convex hull of its extreme vertices.
\end{proof}

\begin{theorem}
	\label{thm:ops-geometries2}
	An oriented graph $D$ is a convex geometry in the $\opthrees$-convexity if and only if $D$ is a $\{\opfour, \oofour\}$-free DAG and for every set $S \subseteq V(D)$ such that $D[S] \in \mathcal{H}$, it holds that $\phullops{H}(S)$ is geometric.
\end{theorem}

\begin{proof}
First, consider that $D$ is a convex geometry in the $\opthrees$-convexity. Let us first prove that $D$ is a DAG. By contradiction, suppose that $D$ contains a directed cycle and let $C=\overrightarrow{C_n}$ be one of minimum length for some integer $n \geq 3$. Thus $C$ is an induced subdigraph of $D$ that has no extreme vertices. Moreover, by Observation \ref{obs:added-non-extreme}, note that $\phullops{D}(V(C))$ is then a non-empty convex set that does not have extreme vertices, a contradiction to the hypothesis that $D$ is a convex geometry in the $\opthrees$-convexity.

Since $D$ is acyclic, then $D$ does not contain a directed cycle on three vertices. This also implies that, for any $S \subseteq V(D)$, a vertex $v \in \phullops{D}(S) \setminus S$ is added to $\phullops{D}(S)$ thanks to the existence of an induced path $(u,v,w)$ such that $v$ is the internal vertex whose endpoints $u$ and $w$ belong to $\phullops{D}(S)$. Thus, neither the arc $(u,w)$ nor $(w,u)$ may belong to $D$.

Let us now prove that $D$ is $\oofour$-free. By contradiction, suppose that $D$ contains an induced subdigraph $O$ isomorphic to $\oofour$. Without loss of generality, suppose that the vertices of $O$ are labeled as in Figure \ref{fig:ofour}.
By Observation~\ref{obs:added-non-extreme}, the set of extreme vertices of $\phullops{D}(V(O))$ is a subset of $\{x_0,x_3\}$. Since $(x_0,x_3) \in A(D)$ and $D$ is acyclic (and thus no directed $(x_3,x_0)$-path exists), observe that $\phullops{D}(\{x_0,x_3\}))=\{x_0,x_3\}$. Consequently, we have that $\phullops{D}(V(O))$ is a convex set in the $\opthrees$-convexity defined by $D$ that is not geometric, a contradiction.

Now, by contradiction, suppose that $D$ contains an induced subdigraph $P$ isomorphic to $\opfour$. Again, assume that the vertices of $P$ are labeled as in Figure \ref{fig:pfour}. 
By Observation~\ref{obs:added-non-extreme}, as in the previous case, the set of extreme vertices $E$ of $\phullops{D}(V(P))$ is a subset of $\{x_0,x_3\}$. It is clear that if $|E| \le 1$, then $\phullops{D}(V(O))$ is not geometric. Therefore, we can assume that $E = \{x_0,x_3\}$.
Moreover, since $D$ is a convex geometry, assume that $\phullops{D}(V(P)) = \phullops{D}(\{x_0,x_3\})$. Let us study which vertices of $D$ belong to $\phullops{D}(\{x_0,x_3\})$. Since $(x_0,x_3)\notin A(D)$, the last argument is not as trivial as before. Let $S = \intfuncops(\{x_0,x_3\})$ be the set of vertices of induced $(x_0,x_3)$-paths of length two in $D$ (recall that we have no directed $(x_3,x_0)$-paths, since $D$ is acyclic). Note that $\phullops{D}(\{x_0,x_3\}) \neq S$, since $x_1$ and $x_2$ do not belong to $S$, which would contradict the hypothesis that $D$ is a convex geometry. Consequently, let $v\in \pintfuncops{2}(\{x_0,x_3\})\setminus \intfuncops(\{x_0,x_3\})$ and let $u$ and $w$ be vertices of $S$ such that $\{u,w\}\subseteq S$ and $v$ is the internal vertex of an induced directed $(u,w)$-path in $D$.

Let us argue that $\{u,w\}\cap \{x_0,x_3\} = \emptyset$. By contradiction, suppose that $\{u,w\}\cap \{x_0,x_3\} \neq \emptyset$ and then that $|\{u,w\}\cap \{x_0,x_3\}| = 1$, since $v\in (\pintfuncops{2}(\{x_0,x_3\})\setminus \intfuncops(\{x_0,x_3\}))$. Without loss of generality, assume that $u\in (\intfuncops(\{x_0,x_3\})\setminus \{x_0,x_3\})$ and thus that $w = x_3$, since $D$ is acyclic. Note that the arc $(u,x_3)$ belongs to $D$, since $u\in (\intfuncops(\{x_0,x_3\})\setminus \{x_0,x_3\})$, which contradicts that the path $(u,v,x_3)$ was induced. Therefore, $\{u,w\}\cap \{x_0,x_3\} = \emptyset$.

Moreover, $x_0$ is not an in-neighbor of $v$ or $x_3$ is not an out-neighbor of $v$, since $v\in \pintfuncops{2}(\{x_0,x_3\})\setminus \intfuncops(\{x_0,x_3\})$. If $x_0$ is not an in-neighbor of $v$, then $\{x_0,u,v,w\}$ induces a $\oofour$ in $D$, contradicting what we proved last. The other case is analogous.

At last, since $D$ is a convex geometry, there is nothing to show for every set $S \subseteq V(D)$ such that $D[S] \in {\cal H}$.

\bigskip

Conversely, consider that $D$ is a $\{\opfour, \oofour\}$-free DAG and for every set $S \subseteq V(D)$ such that $D[S] \in {\cal H}$, it holds that $\hull(S)$ is geometric.
Suppose by contradiction that $D$ is not a convex geometry, i.e., $D$ has an induced subdigraph $G$ such that $V(G)$ is a convex set in the $\opthrees$-convexity defined over $D$ that is not geometric. We can assume that $D$ and $V(G)$ are minimal, i.e., any proper subdigraph of $D$ is a convex geometry and any proper subset of $V(G)$ is geometric.
These assumptions imply that $D = G$. Denote $B = \phullops{D}(\pextps{D}(V(D)))$ and $C = V(D) \setminus B$.
We will show that $C$ can be partitioned into $C_{\ell} \cup C_r$ such that $D$ is a member of $\mathcal{H}$.

We know that every vertex $c \in C$ is the central vertex of an induced $P_3$ of $D$ because $c \not\in \pextps{D}(V(D))$. Furthermore, we claim that for every $c \in C$, there is an induced path $(u,c,v)$ such that $|B \cap \{u,v\}| = 1$.
We know that in such a path it holds $|B \cap \{u,v\}| \le 1$ because if $u,v \in B$, then $c$ would belong to $B$ and not to $C$. Then, suppose by contradiction that for some $c \in C$, every induced path $(u,c,v)$ satisfies $u,v \in C$.

Since $D$ is a DAG, we can choose $c$ such that no ancestor $c'$ of $c$ is such that every induced path $(u,c',v)$ satisfies $u,v \in C$.
Therefore, there is an induced path $(b,u,u')$ or $(u',u,b)$ such that $b \in B$ and $u$ is an ancestor of $c$.
In the former case, since $D$ is acyclic $\{\oofour,\opfour\}$-free, $u' \ne c$ and $(b,c) \in A(D)$.
Since $c$ is not the central vertex of an induced dipath on three vertices containing a vertex of $b$, it holds that $(b,v) \in A(D)$.
Since $v \in C$, there is an induced path $(v'',v,v')$ such that $v'' \ne c$ and $(c,v') \in A(D)$, as otherwise $D$ would contain $\oofour$ or $\opfour$ as an induced subdigraph.
Since $c$ is not the central vertex of an induced dipath on three vertices containing a vertex of $B$, it holds that $(b,v') \in A(D)$.
Let $D' = D - (b,v)$.
The latter case is analogous and implies that $(u',v) \in A(D)$. 
In this case, let $D' = D - (u',v)$.

In both cases, Observation~\ref{obs:deleteArc} implies that $V(D')$ and $B$ are convex sets in $D'$. However, the deletion of an arc $(x,y)$ of a DAG $F$ can change $\pextps{F}(V(F))$ by making $x$ to become a sink, or $y$ to become a source, or a transitive vertex $z \not\in \{x,y\}$ to become a non-transitive vertex.
Since $(v'',v,v')$ is an induced path of $D'$, $v$ is not an extreme vertex of $D'$. Hence
$\pextps{D'}(V(D')) \subseteq \pextps{D}(V(D)) \cup \{b\}$.
Since $b \in B$, $\phullops{D'}(\pextps{D'}(V(D'))) = \phullops{D}(\pextps{D}(V(D))) = B$, which contradicts the assumption that $D$ is a minimal digraph that ha a convex set that is not geometric. Therefore, the claim is true.

Let $c$ be any vertex of $C$. First, consider that $c$ is the central vertex of an induced path $(b,c,c')$ for some $b \in B$ and $c' \in C$.
The case where we begin considering $c$ as the central vertex of an induced path $(c',c,b)$ for some $b \in B$ and $c' \in C$ is symmetric. Set $\ell' = 1$, $r' = 0$ and $c_\ell^{\ell'} = c$.

From now on, repeat the following process for $c'$ until that $c'$ be a vertex of $C_{\ell} \cup C_r$.
Since $c' \in C$, there is an induced path $(u,c',v)$ such that $|B \cap \{u,v\}| = 1$.
If $u \in B$, then set $\ell' = \ell'+1$, set $c^{\ell'+r'} = c'$, add $c'$ to $C_\ell$, redefine $c'$ as $v$ and repeat.
Otherwise, set $r' = r'+1$, set $c^{\ell'+r'} = c'$, add $c'$ to $C_r$, redefine $c'$ as $u$ and repeat.

Since $C$ is finite, this process eventually finishes.
By the construction, $C_{\ell} \cap C_r = \emptyset$ and it holds that no vertex of $C$ is the internal vertex of an induced path with both endpoints in $B$.
In fact, the minimality of $D$ implies that $C = C_{\ell} \cup C_r$ and that $|B| \le |C_{\ell}| + |C_r|$.
Furthermore, we claim that the process of the construction of the sets $C_{\ell}$ and $C_r$ finishes when $c'$ receives the value of $c^1$. Then, suppose by contradiction that the process finishes when $c'$ receives the value of $c^i$ for $i > 1$. Then, observe that $D[B \cup \{c^i, \ldots, c^{\ell'+r'} \}]$ is a graph that is not a convex geometry with fewer vertices than $D$, which is a contradiction.

The above facts imply that Items~\ref{ite:non-empty} to~\ref{ite:Crnoextreme} do hold.
It remains to show that Item \ref{ite:t} also holds.
Consider that $c_1 \in C_\ell$. The case where $c_1 \in C_r$ is analogous. Then, there is an induced path $(b_1,c_1,c_2)$ for $b_1 \in B$ and $c_2 \in C$.
If $c_2 \in C_{\ell}$, then there is an induced path $(b_2,c_2,c_3)$ for $b_2 \in B$ and $c_3 \in C$.
The case where $c_3 \in C_{\ell}$ is depicted in Figure~\ref{fig:item1}$(a)$ and
the case where $c_3 \in C_r$ is depicted in Figure~\ref{fig:item1}$(b)$.
Since $D$ is $\{\opfour, \oofour\}$-free, it holds that $b_2 \ne b_1$, which means that $|B| \ge 2$.
We also have that $c_3 \ne c_1$, which implies that $\ell' + r' \ge 3$. The case where $c_2 \in C_r$ is considered in Figure~\ref{fig:item1}$(c)$ and~$(d)$.
In all cases, it holds that there is an induced path $(b_3,c_3,c_4)$ for $b_3 \in B$ and $c_4 \in C$.
Note that $c_4 \not\in \{c_1,c_2\}$, which implies that $\ell' + r' \ge 4$.

Now, suppose by contradiction that $r' = 1$. The proof for the case where $\ell' = 1$ is analogous. By the circular order of $C_\ell \cup C_r$, we can assume that the only vertex of $C_r$ is $c^{\ell'+r'}$.
By the construction of $C_{\ell} \cup C_r$, for every positive $i \le \ell'+r'-1$, there is $b_i \in B$ such that $(b_i,c^i,c^{i+1})$ is an induced path. There is also a $b$ such that $(c^1,c^{\ell'+r'},b)$ is an induced path.
Since $D$ is acyclic and $\{\opfour, \oofour\}$-free,
the induced path $(b_{\ell'+r'-1}, c^{\ell'+r'-1}, c^{\ell'+r'})$ implies that the edge $(c^{\ell'+r'-1},b)$ exists.
If $\ell'+r' \ge 3$, the induced path $(b_{\ell'+r'-2}, c^{\ell'+r'-2}, c^{\ell'+r'-1})$ implies that the edge $(c^{\ell'+r'-2},b)$ exists.
In fact, for any $i \le \ell'+r'-1$, we have that the edge $(c^i,b)$ exists. However, this is a contradiction because of the induced path $(c^1,c^{\ell'+r'},b)$.
\end{proof}

\begin{figure}

%Cambie \def porque la revista no le gusta
\begin{tikzpicture}

\usetikzlibrary{arrows.meta}

\tikzset{texto/.style={circle,  minimum size=0pt, inner sep=0pt}}

%\begin{scope}[every node/.style={draw, fill=white, circle, thick, fill = white, inner sep=2pt}]

\begin{scope}[every node/.style={draw, fill=white, circle, thick, fill = white, inner sep=2pt}, every edge/.style={draw=black,thick}]

% LLL
\pgfmathsetmacro\x{0};
\pgfmathsetmacro\y{0};

\node [texto] (I) at (0, 1 ) [label=below:$(a)$]{};

\node [black] (b1) at (0 +\x,-1 +\y) [label=below:$b_1$]{};
\node [black] (b2) at (1.5 +\x,-1 +\y) [label=below:$b_2$]{};
\node [black] (b3) at (2 +\x,-2.5 +\y) [label=below:$b_3$]{};

\node (c1) at (1 +\x,0 +\y) [label=above:$c_1$]{};
\node (c2) at (3 +\x,0 +\y) [label=above:$c_2$]{};
\node (c3) at (4 +\x,-1 +\y) [label=above:$c_3$]{};
\node (c4) at (6 +\x,-1 +\y) [label=below:$c_4$]{};

\path [-Stealth] (b1) edge (c1);
\path [-Stealth,dashed] (b1) edge (c2);
\path [-Stealth] (c1) edge [bend right=-20] (c2);

\path [-Stealth] (b2) edge [bend right=20] (c2);
\path [-Stealth,dashed] (b2) edge (c3);
\path [-Stealth] (c2) edge (c3);

\path [-Stealth] (b3) edge (c3);
\path [-Stealth,dashed] (b3) edge [bend right=20] (c4);
\path [-Stealth] (c3) edge (c4);

% LLR (b)
\pgfmathsetmacro\x{8};
\pgfmathsetmacro\y{0};

\node [texto] (I) at (\x, 1+\y ) [label=below:$(b)$]{};

\node [black] (b1) at (0 +\x,-1 +\y) [label=below:$b_1$]{};
\node [black] (b2) at (1.5 +\x,-1 +\y) [label=below:$b_2$]{};
\node (c4) at (2 +\x,-2 +\y) [label=below:$c_4$]{};

\node (c1) at (1 +\x,0 +\y) [label=above:$c_1$]{};
\node (c2) at (3 +\x,0 +\y) [label=below:$c_2$]{};
\node (c3) at (4 +\x,-1 +\y) [label=above:$c_3$]{};
\node [black] (b3) at (6 +\x,-1 +\y) [label=below:$b_3$]{};

\path [-Stealth] (b1) edge (c1);
\path [-Stealth,dashed] (b1) edge (c2);
\path [-Stealth] (c1) edge [bend right=-20] (c2);

\path [-Stealth] (b2) edge [bend right=20] (c2);
\path [-Stealth,dashed] (b2) edge (c3);
\path [-Stealth] (c2) edge (c3);

\path [-Stealth] (c4) edge [bend right=10] (c3);
\path [-Stealth,dashed] (c4) edge [bend right=20] (b3);
\path [-Stealth] (c3) edge (b3);

% LRL (c)
\pgfmathsetmacro\x{0};
\pgfmathsetmacro\y{-5};

\node [texto] (I) at (\x, 1+\y ) [label=below:$(c)$]{};

\node [black] (b1) at (0 +\x,-1 +\y) [label=below:$b_1$]{};
\node  (c3) at (1.5 +\x,-1 +\y) [label=below:$c_3$]{};
\node [black] (b3) at (\x,-2 +\y) [label=below:$b_3$]{};

\node (c1) at (1 +\x,0 +\y) [label=above:$c_1$]{};
\node (c2) at (3 +\x,0 +\y) [label=above:$c_2$]{};
\node [black] (b2) at (5 +\x,\y) [label=above:$b_2$]{};
\node (c4) at (6 +\x,-1 +\y) [label=below:$c_4$]{};

\path [-Stealth] (b1) edge (c1);
\path [-Stealth,dashed] (b1) edge (c2);
\path [-Stealth] (c1) edge [bend right=-20] (c2);

\path [-Stealth] (c3) edge [bend right=20] (c2);
\path [-Stealth,dashed] (c3) edge [bend right=15] (b2);
\path [-Stealth] (c2) edge (b2);

\path [-Stealth] (b3) edge (c3);
\path [-Stealth,dashed] (b3) edge [bend right=25] (c4);
\path [-Stealth] (c3) edge (c4);

% LRR
\pgfmathsetmacro\x{8};
\pgfmathsetmacro\y{-5};

\node [texto] (I) at (\x, 1+\y ) [label=below:$(d)$]{};

\node [black] (b1) at (0 +\x,-1 +\y) [label=below:$b_1$]{};
\node (c3) at (1.5 +\x,-1 +\y) [label=below:$c_3$]{};
\node (c4) at (\x,-2 +\y) [label=below:$c_4$]{};

\node (c1) at (1 +\x,0 +\y) [label=above:$c_1$]{};
\node (c2) at (3 +\x,0 +\y) [label=above:$c_2$]{};
\node [black] (b2) at (5 +\x,\y) [label=above:$b_2$]{};
\node [black] (b3) at (6 +\x,-1 +\y) [label=below:$b_3$]{};

\path [-Stealth] (b1) edge (c1);
\path [-Stealth,dashed] (b1) edge (c2);
\path [-Stealth] (c1) edge [bend right=-20] (c2);

\path [-Stealth] (c3) edge [bend right=20] (c2);
\path [-Stealth,dashed] (c3) edge [bend right=15] (b2);
\path [-Stealth] (c2) edge (b2);

\path [-Stealth] (c4) edge (c3);
\path [-Stealth,dashed] (c4) edge [bend right=25] (b3);
\path [-Stealth] (c3) edge (b3);

\end{scope}

\end{tikzpicture}    

\caption{$(a)$ Case where $c_1, c_2, c_3 \in C_\ell$.
$(b)$ Case where $c_1, c_2 \in C_\ell$ and $c_3 \in C_r$.
$(c)$ Case where $c_1, c_3 \in C_\ell$ and $c_2 \in C_r$.
$(d)$ Case where $c_1 \in C_\ell$ and $c_2, c_3 \in C_r$.
\label{fig:item1}
}

\end{figure}
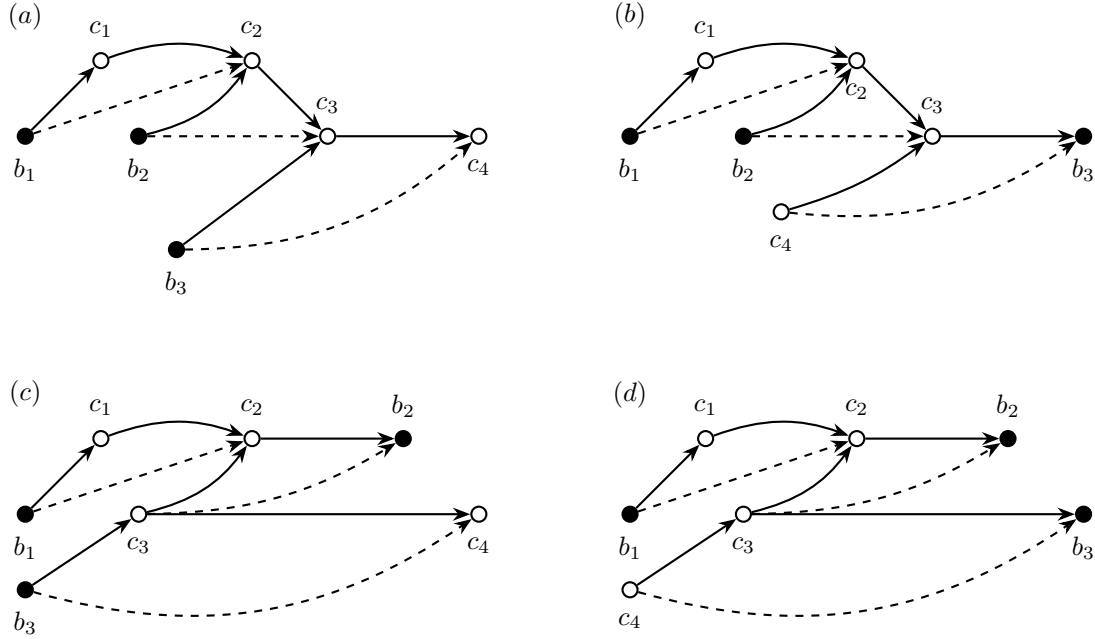

%%%%%%%%%%%%%%%%%%%%%%%%%%%%%%%%%%%%%%%%%%%
\subsection{Recognition is co-\NP-complete}
\label{sec:coNP}

\begin{theorem}
Given a directed acyclic graph $G$, it is co-\NP-complete to decide whether $G$ is a convex geometry in the $\opthrees$-convexity.
\end{theorem}

\begin{proof}
Let $G$ be an oriented graph. By Proposition~\ref{prop:extremep3s}, the set of extreme vertices of any convex subset of $V(G)$ can be found in polynomial time. Given three vertices $u,v$ and $w$; the vertex $v$ belongs to the $\opthrees$-interval of $\{u,w\}$ if and only if the arcs $(u,v)$ and $(v,w)$ exist while the arc $(u,w)$ does not exist. It is clear that the existence of these arcs can be checked in polynomial time. Since $G$ is finite, the convex hull in the $\opthrees$-convexity of any subset of $V(G)$ can be computed in polynomial time. Therefore, given a set $S \subseteq V(G)$, one can confirm that $S$ is convex and not geometric in polynomial time, which implies that the problem of deciding whether a given oriented graph is a convex geometry in the $\opthrees$-convexity belongs to co-\NP.

For the hardness part, we present a reduction from the {\sc 3-SAT} problem. Consider a boolean expression $\mathcal{E}$ in the conjunctive normal form with $m$ clauses $C_1, \ldots, C_m$ on $n$ variables $a_1, \ldots, a_n$. For every clause $C_i$ of $\mathcal{E}$, denote its three literals by $\ell_{i,1}, \ell_{i,2}$ and $\ell_{i,3}$. We can assume that there is no clause of $\mathcal{E}$ containing two or more occurrences of the same variable. Furthermore, we can assume that there is no $i \in [m-1]$ such that $C_i$ contains a literal, which is the negation of a literal of $C_{i+1}$. Indeed, if this occurs, then we can create three variables $a_{n+1},a_{n+2},a_{n+3}$ and insert a copy of the clause $(a_{n+1} \vee a_{n+2} \vee a_{n+3})$ in any necessary place.
Let $\mathcal{R}$ be the family of pairs of literals of ${\cal E}$ belonging to different clauses such that one is not the negation of the other; and let ${\cal T}$ be the family of pairs of literals of ${\cal E}$ belonging to different clauses such that one is the negation of the other.

We construct an oriented graph $G$ as follows. See Figure \ref{fig:reduction} for an example.

\begin{itemize}
\item For every clause $C_i \in {\cal E}$, do the following (the arcs described below are black in Figure~\ref{fig:reduction}):

\begin{itemize}
\item add vertices $x_i$ and $y_i$, and add the arc $(x_i,y_i)$;

\item for every $j \in [3]$, add vertices $w_{i,j}$ and $z_{i,j}$, and add the arcs $(w_{i,j},x_i), (x_i, z_{i,j})$ and $(y_i, z_{i,j})$; and

\item for every $j \in [3]$, add the arc $(w_{i,j},z_{i,j})$.

\end{itemize}

\item For every pair of literals $\{\ell_{p,q}, \ell_{r,s}\} \in {\cal R}$, add the arcs $(w_{p,q}, z_{r,s})$ and $(w_{r,s}, z_{p,q})$. They are green arcs in Figure~\ref{fig:reduction}.

\item Consider an ordering of the members of ${\cal T}$ and denote its quantity by $t$.
For every $t' \in [ t ]$, let $\{\ell_{p,q}, \ell_{r,s}\}$ be the $t'$-th member ${\cal T}$. Then, add vertices $w'_{t'},$$ x'_{t'},$$ y'_{t'},$$ z'_{t'}$, and add the arcs $(w'_{t'},x'_{t'}),$$ (x'_{t'},y'_{t'}),$ $(y'_{t'},z'_{t'}),$$ (w'_{t'},z'_{t'})$ and $(x'_{t'},z'_{t'})$, they are black arcs in Figure~\ref{fig:reduction}, and add the arcs
$(w_{p,q},x'_{t'})$, $(w_{p,q},z'_{t'})$, $(x'_{t'},z_{r,s})$, $(w'_{t'},z_{r,s})$,
$(w_{r,s},x'_{t'})$, $(w_{r,s},z'_{t'})$, $(x'_{t'},z_{p,q})$ and $(w'_{t'},z_{p,q})$. They are red arcs in Figure~\ref{fig:reduction}.

The following edges are blue in Figure~\ref{fig:reduction}.

\item For every $i \in [m - 1]$ and every $j,k \in [3]$, add the arcs $(w_{i,j},y_{i+1}),$ $(w_{i,j},z_{i+1,k}),$$ (x_i,y_{i+1})$.

\item For every $j \in [3]$, add the arcs $(w_{m,j},y'_1)$, $(w_{m,j},z'_1),$$ (x_m,y'_1)$.

\item For every $i \in [t - 1]$, add the arcs $(w'_i,y'_{i+1}),$$ (w'_i,z'_{i+1}), (x'_i,y'_{i+1})$.

\item For every $j \in [3]$, add the arcs $(w'_t,y_1),$$ (w'_t,z_{1,j}),$$ (x'_t,y_1)$.

\end{itemize}

Denote
$W' = \{w'_i : i \in [t]\}$,
$X' = \{x'_i : i \in [t]\}$,
$Y' = \{y'_i : i \in [t]\}$,
$Z' = \{z'_i : i \in [t]\}$,
$W = W' \cup \{w_{i,j} : i \in [m]$ and $ j \in [3]\}$,
$X = X' \cup \{x_i : i \in [m] \}$,
$Y = Y' \cup\{y_i : i \in [m] \}$ and
$Z = Z' \cup \{z_{i,j} : i \in [m]$ and $ j \in [3]\}$.

hardness reduction
\begin{figure}[htbp]
\begin{center}
\usetikzlibrary{arrows, shapes,backgrounds,arrows.meta, shapes, positioning}

\begin{tikzpicture}[xscale=0.7]

\begin{scope}[every node/.style={draw, fill=white, circle, thick, inner sep=23pt}]
        \node (W1) at (-4.5-0.4,2) []{};
        \node (W2) at (-4.5-0.4,-2.5) []{};
        \node (W3) at (-4.5-0.4,-7) []{};

        \node (Z1) at (7.5+0.4,2) []{};
        \node (Z2) at (7.5+0.4,-2.5) []{};
        \node (Z3) at (7.5+0.4,-7) []{};
\end{scope}
    
\begin{scope}[every node/.style={draw, fill=white, circle, thick, fill = white, inner sep=2pt}]

% vertices of W
        \node (w11) at (-4.5,2+0.5) [label=left:$w_{1,1}$]{};
        \node (w12) at (-4.5,2) [label=left:$w_{1,2}$]{};
        \node (w13) at (-4.5,2-0.5) [label=left:$w_{1,3}$]{};

        \node (w21) at (-4.5,-2.5+0.5) [label=left:$w_{2,1}$]{};
        \node (w22) at (-4.5,-2.5) [label=left:$w_{2,2}$]{};
        \node (w23) at (-4.5,-2.5-0.5) [label=left:$w_{2,3}$]{};

        \node (w31) at (-4.5,-7+0.5) [label=left:$w_{3,1}$]{};
        \node (w32) at (-4.5,-7) [label=left:$w_{3,2}$]{};
        \node (w33) at (-4.5,-7-0.5) [label=left:$w_{3,3}$]{};

        \node (w'1) at (-4.5,-11) [label=left:$w'_1$]{};
        \node (w'2) at (-4.5,-11-3) [label=left:$w'_2$]{};

% vertices of X
        \node (x1) at (0,2) [label=above:$x_1$]{};
        \node (x2) at (0,-2.5) [label=above:$x_2$]{};
        \node (x3) at (0,-7) [label=above:$x_3$]{};

        \node (x'1) at (0,-11) [label=above:$x'_1$]{};
        \node (x'2) at (0,-11-3) [label=below:$x'_2$]{};

% vertices of Y
        \node (y1) at (3,2) [label=above:$y_1$]{};
        \node (y2) at (3,-2.5) [label=above:$y_2$]{};
        \node (y3) at (3,-7) [label=above:$y_3$]{};

        \node (y'1) at (3,-11) [label=above:$y'_1$]{};
        \node (y'2) at (3,-11-3) [label=above:$y'_2$]{};

% vertices of Z
        \node (z11) at (7.5,2+0.5) [label=right:$z_{1,1}$]{};
        \node (z12) at (7.5,2) [label=right:$z_{1,2}$]{};
        \node (z13) at (7.5,2-0.5) [label=right:$z_{1,3}$]{};

        \node (z21) at (7.5,-2.5+0.5) [label=right:$z_{2,1}$]{};
        \node (z22) at (7.5,-2.5) [label=right:$z_{2,2}$]{};
        \node (z23) at (7.5,-2.5-0.5) [label=right:$z_{2,3}$]{};

        \node (z31) at (7.5,-7+0.5) [label=right:$z_{3,1}$]{};
        \node (z32) at (7.5,-7) [label=right:$z_{3,2}$]{};
        \node (z33) at (7.5,-7-0.5) [label=right:$z_{3,3}$]{};

        \node (z'1) at (7.5,-11) [label=right:$z'_1$]{};
        \node (z'2) at (7.5,-11-3) [label=right:$z'_2$]{};
\end{scope}
      
\begin{scope}[>=stealth,every edge/.style={draw=black}]

% edges W --> X
        \path [->] (W1) edge node  {} (x1);
        \path [->] (W2) edge node  {} (x2);
        \path [->] (W3) edge node  {} (x3);

        \path [->] (w'1) edge node  {} (x'1);
        \path [->] (w'2) edge node  {} (x'2);

% edges W --> Z
        \path [->] (w11) edge [bend right=-25] node  {} (z11);
        \path [->] (w12) edge [bend right=-27] node  {} (z12);
         \path [->] (w13) edge [bend right=-30] node  {} (z13);

        \path [->] (w21) edge [bend right=-25] node  {} (z21);
        \path [->] (w22) edge [bend right=-27] node  {} (z22);
       \path [->] (w23) edge [bend right=-30] node  {} (z23);

        \path [->] (w31) edge [bend right=-25] node  {} (z31);
        \path [->] (w32) edge [bend right=-27] node  {} (z32);
        \path [->] (w33) edge [bend right=-30] node  {} (z33);

        \path [->] (w'1) edge [bend right=-20] node  {} (z'1);
        \path [->] (w'2) edge [bend right=-20] node  {} (z'2);

% edges X --> Y
        \path [->] (x1) edge node  {} (y1);
        \path [->] (x2) edge node  {} (y2);
        \path [->] (x3) edge node  {} (y3);

        \path [->] (x'1) edge node  {} (y'1);
        \path [->] (x'2) edge node  {} (y'2);

% edges X --> Z
        \path [->] (x1) edge [bend right=20] node  {} (Z1);
        \path [->] (x2) edge [bend right=20] node  {} (Z2);
        \path [->] (x3) edge [bend right=20] node  {} (Z3);

        \path [->] (x'1) edge [bend right=15] node  {} (z'1);
        \path [->] (x'2) edge [bend right=15]  node  {} (z'2);

% edges Y --> Z
        \path [->] (y1) edge node  {} (Z1);
        \path [->] (y2) edge node  {} (Z2);
        \path [->] (y3) edge node  {} (Z3);

        \path [->] (y'1) edge node  {} (z'1);
        \path [->] (y'2) edge node  {} (z'2);
      \end{scope}

\begin{scope}[>=stealth,every edge/.style={draw=green}]

% edges not negated
      \path [->] (w11) edge node  {} (Z2);
      \path [->] (W2) edge node  {} (z11);
      \path [->] (w11) edge node  {} (z32);
      \path [->] (w32) edge node  {} (z11);
      \path [->] (w11) edge node  {} (z33);
      \path [->] (w33) edge node  {} (z11);
\end{scope}

\begin{scope}[>=stealth,every edge/.style={draw=red}]

% edges negated pair 1
      \path [->] (w11) edge node  {} (x'1);
      \path [->] (w11) edge node  {} (z'1);
      \path [->] (x'1) edge node  {} (z31);
      \path [->] (w'1) edge node  {} (z31);
      \path [->] (w31) edge node  {} (x'1);
      \path [->] (w31) edge node  {} (z'1);
      \path [->] (x'1) edge node  {} (z11);
      \path [->] (w'1) edge node  {} (z11);
\end{scope}

\begin{scope}[>=stealth,every edge/.style={draw=blue}]

% edges W --> Y
      \path [->] (W1) edge node  {} (y2);
      \path [->] (W2) edge node  {} (y3);
      \path [->] (w'1) edge node  {} (y'2);

% edges X --> Y
      \path [->] (x1) edge node  {} (y2);
      \path [->] (x2) edge node  {} (y3);
      \path [->] (x3) edge node  {} (y'1);

      \path [->] (x'1) edge node  {} (y'2);
      \path [->] (x'2) edge node  {} (y1);

% edges W --> Z
      \path [->] (W1) edge node  {} (Z2);
      \path [->] (W2) edge node  {} (Z3);
      \path [->] (w'1) edge node  {} (z'2);
      \path [->] (w'2) edge node  {} (y1);
      \path [->] (w'2) edge node  {} (Z1);
\end{scope}

\end{tikzpicture}    

\caption[]{Oriented graph $G$ constructed from the boolean expression $\mathcal{E}=(a_1 \vee\overline{a_2}\vee a_3)\wedge (\overline{a_2}\vee a_4 \vee a_5) \wedge (\overline{a_1}\vee \overline{a_3}\vee a_4)$. Family $\mathcal{T} = \{ \{\ell_{1,1},\ell_{3,1}\}, \{\ell_{1,2},\ell_{4,1}\} \}$ and $t=2$.
To make the drawing cleaner, we do not draw all arcs. We used an arc from a set of vertices $Q$ to a set $Q'$ standing that there is an arc from every vertex of $Q$ to every vertex of $Q'$. All black and blue edges are drawn. For the green and red edges, only the ones associated with literal $a_1$ are drawn.}

\label{fig:reduction}
\end{center}	
\end{figure}
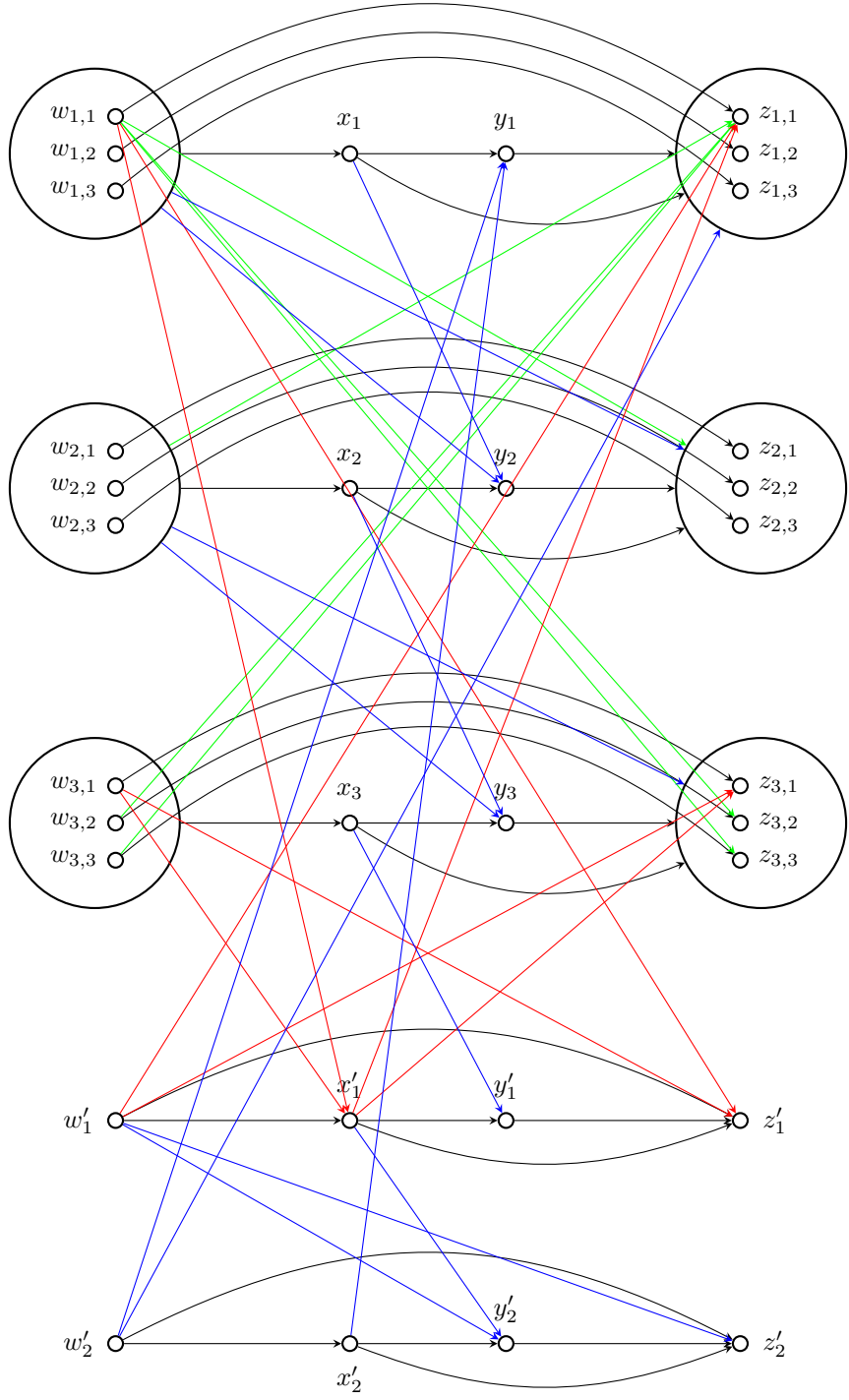

We will show that there is a truth assignment that satisfies all clauses of ${\cal E}$ if and only if $G$ is not a convex geometry in the $\opthrees$-convexity.

\bigskip

For the necessity, consider that there is a truth assignment $\phi$ that satisfies all clauses of ${\cal E}$. For every $i \in [m]$, denote by $\ell_{i,j_i}$ one of the literals of $C_i$ that has value true under $\phi$. Construct a set $S \subset V(G)$ as follows. For every $i \in [m]$, add $w_{i,j_i}$ and $z_{i,j_i}$ to $S$; for $i \in [t]$, add $w'_i$ and $z'_i$ to $S$; and add all vertices of $X \cup Y$ to $S$.

We claim that $S$ is a convex set. By contradiction, suppose that there are $a,c \in S$ and $b \not\in S$ such that $(a,b,c)$ is an induced directed path. By the construction of $G$, every vertex of $W$ is a source of $G$ and every vertex of $Z$ is a sink of $G$, which implies that $b \in X \cup Y$. But this yields a contradiction because $X \cup Y \subset S$. Therefore, $S$ is a convex set.

Next, we show that $S \cap (W \cup Z) = \extps(S)$.
Every vertex of $X \cup Y \subset S$ is not an extreme vertex of $S$ because of the following facts.

\begin{itemize}

\item for every $i \in [m]$, $x_i$ is not extreme in $S$ because $y_i \in S$, there is $w_{i,j} \in S$ for some $j \in [3]$ and $(w_{i,j},x_i,y_i)$ is an induced directed path;

\item the vertex $y_1$ is not extreme in $S$, because $x'_t \in S$ and there is $z_{1,j} \in S$ for some $j \in [3]$, and the dipath $(x'_t, y_1, z_{1,j})$ is induced;

\item for every $i \in \{2, \ldots, m\}$, $y_i$ is not extreme in $S$ because $x_{i-1} \in S$, there is $z_{i,j} \in S$ for some $j \in [3]$ and $(x_{i-1},y_i,z_{i,j})$ is an induced dipath;

\item for every $i \in [t]$, $(w'_i,x'_i,y'_i)$ is an induced dipath and $w'_i,y'_i \in S$;

\item the vertex $y'_1$ is not extreme in $S$ because $x_m,z'_1 \in S$ and the dipath $(x_m, y'_1, z'_1)$ is induced; and

\item for every $i \in \{2, \ldots, t\}$, $(x'_{i-1},x'_i,z'_i)$ is an induced dipath, and $x'_{i-1},z'_i \in S$.

\end{itemize}

Hence, since every vertex of $W$ is a source of $G$ and every vertex of $Z$ is a sink of $G$, we have that $S \cap (W \cup Z) = \extps(S)$.

Now, we claim that $\hullops(\extps(S)) = \extps(S)$. By contradiction, suppose the contrary, i.e., there are $w,z \in \extps(S)$ and $u \not\in \extps(S)$ such that $(w,u,z)$ is an induced dipath. By the previous arguments, it holds that $w \in S \cap W$, $z \in S \cap Z$, and $u \in X \cup Y$.

First, consider that $u=x_i$ for some $i \in [m]$.
Since $N^-(x_i)=\{w_{i,1}$,$w_{i,2},$ $w_{i,3}\}$ and $N^+(x_i)=\{z_{i,1},$ $z_{i,2},$ $z_{i,3},y'\}$ for some $y' \in Y$, we have that $w = w_{i,j}$ for some $j \in [3]$, and $z = z_{i,j'}$ for some $j' \in [3]$, which is not possible because $(w_{i,j},z_{i,j'})$ is an arc of $G$.

Now, consider that $u=x'_i$ for some $i \in [t]$. By the construction of $G$, every out-neighbor of $x'_i$ which belongs to $Z$ is also an out-neighbor of $w'_i$. Then we can assume that $w = w_{i,j}$ for some $i \in [m]$ and $j \in [3]$. We also have that every in-neighbor of $x'_i$ belonging to $W$ is also an in-neighbor of $z'_i$. Then we can assume that $z = z_{i',j'}$ for some $i' \in [m]$ and $j' \in [3]$. Since $(w,u,z)$ is an induced path, $(w_{i,j}, z_{i',j'}) \not\in E(G)$, which means that the literals corresponding to $w_{i,j}$ and $z_{i',j'}$ is a pair of ${\cal T}$, i.e., they belong to different clauses and one is the negation of the other, which contradicts the construction of $S$.

Next, consider that $u = y_i$ for some $i \in [m]$. Then, 
$w = w'_t$ if $i = 1$ or $w = w_{i-1,j}$ for some $j \in [3]$ if $i \ge 2$,
and $z = z_{i,j'}$ for some $j' \in [3]$. But note that $(w,z_{i,j'})$ is an arc of $G$, which is not possible.

At last, consider that $u = y'_i$ for some $i \in [t]$.
For $i = 1$, note that $y'_i$ has only three in-neighbors belonging to $W$, namely $w_{m,1}, w_{m,2}$ and $w_{m,3}$, and has only one out-neighbor belonging to $Z$, the vertex $z'_1$. Moreover, we have that $(w_{m,j}, z'_1) \in E(G)$ for every $j \in [3]$, which is not possible.
For every $i \ge 2$, note that $w'_{i-1}$ is the onlyl in-neighbor of $y'_i$ belonging to $W$, and $z'_i$ is the only out-neighbor of $y'_i$ belonging to $Z$. However, $(w'_{i-1}, z'_i) \in E(G)$, which implies that $G$ is not a convex geometry.

\bigskip

For the sufficiency, consider that $G$ is not a convex geometry.
We begin by showing that $G$ is $\{\opfour, \oofour\}$-free.
Note that $W, X, Y$ and $Z$ are independent sets. Note also 
that there is no arc from a vertex of $X \cup Y \cup Z$ to a vertex of $W$,
that there is no arc from a vertex of $Y \cup Z$ to a vertex of $X$,
and that there is no arc from a vertex of $Z$ to a vertex of $Y$.
Therefore, if $G$ has an induced $\opfour$ or an induced $\oofour$ with arcs $(w,x), (x,y)$ and $(y,z)$, we would have that
$w \in W$,
$x \in X$,
$y \in Y$, and
$z \in Z$.
Now, note that among the out-neighbors of every $x' \in X$, two belong to $Y$, say $y'$ and $y''$.
Note that one of these two vertices, let us say $y'$, has as in-neighbors all in-neighbors of $x'$.
Therefore, $x'$ and $y'$ do not belong to an induced $\opfour$ nor an induced $\oofour$.
Note also that $x'$ is an in-neighbor of all out-neighbors of $y''$.
Therefore, $x'$ and $y''$ do not belong to an induced $\opfour$ nor an induced $\oofour$, which means that $G$ is $\{\opfour, \oofour\}$-free.

Therefore, by Theorem~\ref{thm:ops-geometries2}, $G$ has an induced subgraph $H \in {\cal H}$ such that $V(H) \setminus \hullops(B') \ne \emptyset$ where $B' = \extps(\hullops(V(H)))$. Let $(B, C_\ell = \{c_\ell^1, \ldots, c_\ell^{\ell'}\},C_r  = \{c_r^1, \ldots, c_r^{r'}\})$ be a partition of $V(H)$ and $c^1, \ldots, c^{\ell'+r'}$ be a circular ordering of the vertices of $C_\ell \cup C_r$ satisfying Definition \ref{def:familyH}.

By Item \ref{ite:Clnoextreme} of Definition \ref{def:familyH}, 
for every $i \in [\ell'+r']$, if $c^i \in C_\ell$, then there is $b \in B$ such that $(b,c^i,c^{i+1})$ is an induced path. Therefore $C_\ell \subset X \cup Y$.
By Item \ref{ite:Crnoextreme} of Definition \ref{def:familyH}, 
for $i \in [\ell'+r']$, if $c^i \in C_r$, then there is $b \in B$ such that $(c^{i+1},c^i,b)$ is an induced path. Therefore $C_r \subset X \cup Y$. 

Now, note that in the subgraph $G[X \cup Y]$, every vertex of $X$ has in-degree zero and out-degree two, while every vertex of $Y$ has in-degree two and out-degree zero. Therefore, using Item  \ref{ite:lrcycle} of Lemma \ref{lem:cons}, we conclude that $C_\ell \subseteq X$ and $C_r \subseteq Y$.

Furthermore, for every $x \in X \cap C_\ell$, among the two out-neighbors of $x$ in the graph $G[X \cup Y]$, only one can be the vertex $y \in Y \cap C_r$ such that $(b,x,y)$ is an induced path for some $b \in B$.
Similarly, for every $y \in Y \cap C_r$, among the two in-neighbors of $y$ in the graph $G[X \cup Y]$, only one can be the vertex $x \in X \cap C_\ell$ such that $(x,y,b)$ is an induced path for some $b \in B$.
Therefore, by the circular ordering of $C_\ell \cup C_r$, we conclude that $C_\ell = X$ and $C_r = Y$.

Since every vertex of $W \cup Z$ is extreme in $G$ and $B \subset W \cup Z$, we have that every vertex of $B$ is an extreme of $\hull(V(H))$, which means that $B = B'$ and $\hull(V(H)) \setminus B \subseteq X \cup Y$.
Since $V(H) \setminus \hullops(B) \ne \emptyset$, using Item \ref{ite:one} of Lemma \ref{lem:cons}, we conclude that $\hullops(\extps(\hullops(V(H)))) = \hullops(B)$ does not have any vertex of $X \cup Y$.

Using Item \ref{ite:Cnoextreme} of Lemma \ref{lem:cons}, we have that every vertex of $C_\ell$ has at least one in-neighbor $w_i$ of $B$ belonging to $W$.
Note that for different positive values $i,j \le m$, we have that $N^-(x_i) \cap N^-(x_j) = \emptyset$, which means that $w_i \ne w_j$ as long as $i \ne j$.
Let $S = \{w_i : 1 \le i \le m \}$.
Analogously, we have that every vertex of $C_r$ has at least one out-neighbor $z_i$ of $B$ belonging to $Z$. Note that for different positive values $i,j \le m$, we have that $N^+(z_i) \cap N^+(z_j) = \emptyset$, which means that $z_i \ne z_j$ as long as $i \ne j$.
Furthermore, for $w_i = w_{i,q}$ and $z_i = w_{i,q'}$, we have that $q = q'$ because if $q \ne q'$, then $(w_{i,q}, z_{i,q'}) \not\in E(G)$ and $(w_{i,q},x_i, z_{i,q'})$ is an induced path, which means that $\hullops(B)$ contains a vertex of $X \cup Y$, which is not true.

To complete the proof, it suffices to show that for any two vertices of $S$, the corresponding literals belong to different clauses and that one is not the negation of the other.
By the construction of $S$, it is clear that any two vertices correspond to literals of different clauses. Then, suppose by contradiction that there are positive values $i,j \le m$ such that the vertices $w_i = w_{p,q} \in S$ and $w_j = w_{r,s} \in S$ correspond to literals, with one being the negation of the other.
Due to the assumption that there is no positive $i \le m-1$ such that $C_i$ contains a literal that is the negation of a literal of $C_{i+1}$, we have that $|p - r| \ge 2$.
Then $\{w_{p,q},w_{r,s}\}$ is a pair of vertices associated with a pair $\{\ell_{p,q},\ell_{r,s}\} \in {\cal T}$. Assume that $\{\ell_{p,q},\ell_{r,s}\}$ is the $k-th$ pair of ${\cal T}$ in the considered ordering to construct $G$.
Note that $(w_{p,q}, x'_k, z_{r,s})$ is an induced path where $w_{p,q}, z_{r,s} \in B$, which means that 
$\hullops(B)$ contains vertices of $X \cup Y$, a contradiction.
\end{proof}

%%%%%%%%%%%%%%%%%%%%%%%%%%%%%%%%%%%%%%%%
\subsection{Acyclic indifference oriented graphs: a polynomial case.}\label{sec:acycleindif}

In this section, we present a new class of oriented graphs, to the best of our knowledge, in which deciding if an oriented graph is a geometry with respect 
to $\opthrees$-convexity is polynomial-time solvable. The definition of this class is inspired by the class of indifference graphs.

A graph $G$ is an \emph{indifference} graph if and only if there is a total order $\prec$ over $V(G)$ such that if $u\prec v\prec w$ and $uw\in E(G)$, then $uv,vw \in E(G)$~\cite{LO1993}. Instead of a total order, we consider a partial order, and the maximal cliques are replaced by maximal transitive tournaments. A related subclass was presented in~\cite{tesisGuada}, where the author considers a tree-like order in which the Hasse diagram of the partial order is a tree.

\begin{definition}
    \label{def:acyclic-indifference}
    Let $D$ be an oriented acyclic graph. We say that $D$ is \textit{acyclic indifference} if there is a partial order $\prec$ on $V(D)$ such that:
    
    \begin{enumerate}[label=\alph*)]
        \item if $(x,y)\in A(D)$, then $x\prec y$; and
        \item if $x$, $y$, $z$ are vertices of $D$ with $x \prec y \prec   z$ and
    $(x,z)\in A(D)$, then $(x,y)\in A(D)$ and $(y, z)\in A(D)$.
    \end{enumerate} 
\end{definition}

See Figure \ref{fig:exampleToAcyclicIndif} for an example of an acyclic indifference oriented graph.
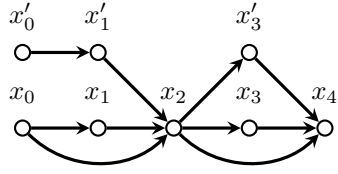
\begin{figure}[!ht]
    \centering
    \begin{tikzpicture}   
      \begin{scope}[every node/.style={draw, fill=white, circle, thick, fill = white, inner sep=2pt}]
        \node (x0) at (-1.5,0) [label=above:$x_0$]{};
        \node (x1) at (-0.5,0) [label=above:$x_1$]{};
        \node (x0p) at (-1.5,1) [label=above:$x_0'$]{};
        \node (x1p) at (-0.5,1) [label=above:$x_1'$]{};
        \node (x2) at (0.5,0) [label=above:$x_2$]{};
        \node (x3) at (1.5,0) [label=above:$x_3$]{};
        \node (x3p) at (1.5,1) [label=above:$x_3'$]{};
        \node (x4) at (2.5,0) [label=above:$x_4$]{};
      \end{scope}
      
      \begin{scope}[>={stealth[black]},
        % 			  every node/.style={fill=white,circle},
        every edge/.style={draw=black,very thick}]
        \path [->] (x0p) edge node  {} (x1p);
        \path [->] (x1p) edge node  {} (x2);
        \path [->] (x0) edge node  {} (x1);
        \path [->] (x1) edge node  {} (x2);
        \path [->] (x2) edge node  {} (x3);
        \path [->] (x3) edge node  {} (x4);
        \path [->] (x2) edge node  {} (x3p);
        \path [->] (x3p) edge node  {} (x4);
        \path [->] (x0) edge [bend right=45] node  {} (x2);
        \path [->] (x2) edge [bend right=45] node  {} (x4);
      \end{scope}
    \end{tikzpicture}    
    \caption[]{Example of an acyclic indifference oriented graph.}
    \label{fig:exampleToAcyclicIndif}
  \end{figure}

An equivalent way to define the graphs in the class is as follows.
\begin{observation}
    \label{obs:acyclicIndif-pathsVersion}
    A DAG $D$ is an \textit{acyclic indifference} if and only if for any vertices $x$, $y$, $z$ of $D$ such that there are an $(x,y)$-path, a $(y,z)$-path, and $(x,z) \in A(D)$, it holds that $(x,y) \in A(D)$ and $(y,z) \in A(D)$.
\end{observation}
    
Note that this class is hereditary and $\oofour$-free.

\begin{theorem}
    \label{thm:acyclic-indifference}
    An acyclic indifference oriented graph $D$ is a convex geometry in the $\opthrees$-    convexity if and only if $D$ is $\overrightarrow{P_4}$-free.
\end{theorem}
\begin{proof}
Let $D$ be an acyclic indifference oriented graph. By definition, recall that $D$ is a DAG and $D$ is $\oofour$-free. If $D$ is a convex geometry in the $\opthrees$-convexity, then $D$ must be $\opfour$-free, by Theorem~\ref{thm:ops-geometries2}.

Conversely, every convex subset of an acyclic indifference oriented graph $D$ that is $\overrightarrow{P_4}$-free induces a subdigraph in the same class, as the family is hereditary, and vertex removal cannot create an induced $\overrightarrow{P_4}$. Thus, it suffices to prove that $V(D)$ is the hull of its extreme vertices. Let $\prec$ be a partial order over $V(D)$ as in Definition \ref{def:acyclic-indifference}.

Since $D$ is acyclic, it has at least one source and at least one sink. Then $V(D)$ has extreme vertices. 
Suppose that $V(D)$ has at least one non-extreme vertex $a$, as otherwise $V(D)$ is trivially the convex hull of its extreme vertices. 

Thus, there exist $x,y\in V(D)$ such that $(x,a),(a,y)\in A(D)$, but $(x,y)\notin A(D)$. Among all maximal transitive tournaments that contain both $x$ and $a$, let $T$ be one whose source $s$ is a minimal element, with respect to the order $\prec$, in the subset of vertices of $V(D)$ composed of all sources of all maximal transitive tournaments containing both $x$ and $a$. Let us prove that $s$ is a source of $D$. 

By contradiction, suppose that it is not the case and let $s'\in V(D)$ be such that $(s',s)\in A(D)$. By the choice of $T$, we claim that $(s',a)\notin T$. Otherwise, by the definition of acyclic indifference oriented graphs, note that $s'$, $s$, $x$, and $a$ would belong to another maximal transitive tournament $T'$ in which $s$ is not a source, contradicting its minimality in $\prec$. Note also that $(s',y)\notin A(D)$ and $(s,y)\notin A(D)$, as $(x,y)\notin A(D)$ and $D$ is an acyclic indifference oriented graph. Consequently, $\{s',s,a,y\}$ induces a $\opfour$, a contradiction. Therefore, $s$ must be a source of $D$.

Analogously, one can prove that there is a sink $t$ of $D$ such that $(a,t)\in A(D)$. Note that $(s,t)\notin A(D)$ as $(x,y)\notin A(D)$ and $D$ is an acyclic indifference oriented graph. Therefore, $a$ is in the convex hull of the extreme vertices of $D$, for any non-extreme vertex $a$ of $D$, which completes the proof.
Thus, by a simple brute-force algorithm to look for induced $\opfour$, one can deduce that deciding whether a given acyclic indifference oriented graph $D$ defines a convex geometry in the $\opthrees$-convexity can be done in $\mathcal{O}(n^4)$-time, $n$ being the number of vertices of $D$.
\end{proof}

Actually, the proof in Theorem \ref{thm:acyclic-indifference} implies that $D$ is a convex geometry if and only if each vertex is either a source, a sink, or lies in an induced $\opthrees$ whose endpoints are a source and a sink of $D$. Therefore, we may deduce that:

\begin{corollary}
    \label{cor:algo-ind-acyclic}
    Deciding whether a given acyclic indifference oriented graph $D$ on $n$ vertices and $m$ arcs defines a convex geometry in the $\opthrees$-convexity can be done in ${\cal O}(n^3)$-time.
\end{corollary}

%%%%%%%%%%%%%%%%%%%%%%%%%%%%%%%%%%%%%%%%
\section{Further research} \label{sec:further}

One should notice that, since the directed paths on three vertices that define the convex sets do not need to be induced/shortest, transitive vertices are no longer extreme vertices in this convexity. One can observe that such oriented graphs must be still acyclic, but our proof of $\oofour$-freeness does not hold in this case. Actually, see Figure \ref{fig:badforP3} for an example of a convex geometry in the $\opthree$ convexity that is not a convex geometry in the $\opthrees$-convexity as it contains an induced directed path on four vertices.

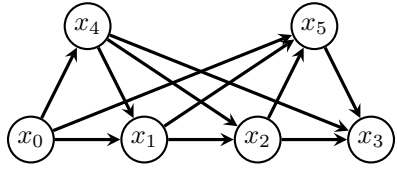
\begin{figure}[!ht]
	\centering
	\begin{tikzpicture}[scale=1.5]
    
      \begin{scope}[every node/.style={draw, fill=white, circle, thick, fill = white, inner sep=2pt}]
        \node (x0) at (-1.5,0) {$x_0$};
        \node (x1) at (-0.5,0) {$x_1$};
        \node (x2) at (0.5,0) {$x_2$};
        \node (x3) at (1.5,0) {$x_3$};
		\node (x4) at (-1,1) {$x_4$};
		\node (x5) at (1,1) {$x_5$};
      \end{scope}
      
      \begin{scope}[>={stealth[black]},
        % 			  every node/.style={fill=white,circle},
        every edge/.style={draw=black,very thick}]
        \path [->] (x0) edge node  {} (x1);
        \path [->] (x1) edge node  {} (x2);
        \path [->] (x2) edge node  {} (x3);
		\path [->] (x0) edge node  {} (x4);
		\path [->] (x4) edge node  {} (x1);
		\path [->] (x4) edge node  {} (x2);
		\path [->] (x4) edge node  {} (x3);
		\path [->] (x0) edge node  {} (x5);
		\path [->] (x1) edge node  {} (x5);
		\path [->] (x2) edge node  {} (x5);
		\path [->] (x5) edge node  {} (x3);
      \end{scope}
    \end{tikzpicture}
    \caption[]{$\opthree$ geometry that is not a $\opthrees$ geometry.}
	\label{fig:badforP3}
\end{figure}

Another interesting direction is to study the $\opthree$-convexity and try to determine which oriented graphs $D$ are convex geometries in the $\opthree$-convexity. One could also try to improve the polynomial-time algorithm to decide whether an oriented graph is a convex geometry in the $\opthree$-convexity, which is a consequence of Theorem \ref{thm:P3-convexity}.

Since deciding whether a directed graph is a convex geometry in the $\opthrees$-convexity is $\coNP$-complete, we can ask whether there are polynomial-time algorithms to decide whether a directed graph is a convex geometry in the $\opthrees$-convexity when it is restricted to particular graph classes. For example, for graphs with bounded treewidth, one can use Courcelle's theorem to decide whether an oriented graph $D$ is a convex geometry in the $\opthrees$-convexity in polynomial time. Indeed, the authors in~\cite{AMM+2026} present MSO formulas to represent the property of $S$ being a convex set and for $S'$ being the convex hull of $S$ in the $\opthrees$-convexity. It is straightforward to propose formulas for the properties of a vertex being extreme, and thus for $D$ being a convex geometry in the $\opthrees$-convexity.

However, the time-complexity of these algorithms is known to be highly exponential in the value of the treewidth. Adhoc algorithms for specific graph classes are still of interest. One can also address this problem for other graph classes, such as bipartite graphs, planar graphs, etc.

%%%%%%%%%%%%%%%%%%%%%%%%%%%%%%%%%%%%%%%%%%

% \section{Declarations}

% The author
% declares that they have no financial interests.

% \section{Data availability}

% No data was used for the research described in the article.

\section*{Acknowledgements}

M. C. Dourado is partially supported by Conselho Nacional de Desenvolvimento Cient\'\i fico e Tecnológico (CNPq), Brazil, Grant numbers 403601/2023-1 and 305141/2024-4. J. Ara\'ujo is partly supported by CNPq projects 313153/2021-3 and 404613/2023-3, Inria Associated Team CANOE and project CAPES-Cofecub 49587PE.


\begin{thebibliography}{00}
\bibitem{ADPS2025}
 J. Ara\'ujo, M. C. Dourado, F. Protti and R. M. Sampaio,
   \textit{Introduction to Graph conexity},
  Springer Cham,
  2025.
\bibitem{AMM+2026}
  J. Ara\'ujo, A. Maia, P. Medeiros and L. Penso,
  \textit{On the hull and interval numbers of oriented graphs (brief announcement)}, Procedia Computer Science 223 (2023), 397-399. XII Latin-American Algorithms, Graphs and Optimization Symposium (LAGOS 2023).

  \bibitem{BG2008} J. Bang Jensen and G. Gutin, \textit{Classes of Digraphs}, Springer Monographs in Mathematics, 2008.

   \bibitem{BM2008} L. Bondy and U. Murty, \textit{Graph Theory}, 1st ed. Springer Publishing Company, Incorporated, 2008.

   \bibitem{DGPST2025} M. C. Dourado, M. Gutierrez, F. Protti, R. M. Sampaio and S. Tondato, \textit{Characterization of graph via convex geometry: A survey}, Discrete Applied Mathematics 360 (2025), 246--257.

   \bibitem{erdHos1972some} P. Erd\"{o}s, E. Fried, A. Hajnal and E. C. Milner, \textit{Some remarks on simple tournaments}, Algebra universalis 2, 1 (1972), 238--245.

   \bibitem{KreinMilman1940} M. Krein and D. Milman, \textit{On extreme points of regular convex sets}, Studia Mathematica 9 (1940), 133--138.

 \bibitem{LO1993} P. J. Looges and S. Olariu,  \textit{Optimal greedy algorithms for indifference graphs}, Computers and Mathematics with Applications 25, 7 (1993), 15–25.527.

 \bibitem{Minkowski1911} H. Minkowski, \textit{Theorie der konvexen korper}, In Gesammelte Abhandlungen, D. Hilbert, Ed., vol. 2 Teubner, Leipzig, 1911, 131--229.

\bibitem{PWW.08} D. B. Parker, R. F. Westhoff and M. J. Wolf, \textit{On two-path convexity in multipartite tournaments}, European J. of Combinatorics 29 (2008), 641–651.

  \bibitem{pfaltz1971convexity} J. L. Pfaltz, \textit{Convexity in directed graphs}, Journal of Combinatorial Theory, Series B 10, 2 (1971), 143--162.

  \bibitem{tesisGuada} M. G. S\'anchez Valluv\'\i , \textit{Operadores de Torneos Transitivos Maximales},
 Tesis doctoral de La Facultad de Ciencias Exactas, Universidad Nacional de La Plata, 2025.

\bibitem{vandeVel1993} M. Van de Vel, \textit{Theory of Convex Structures}, ISSN. Elseiver Science, 1993.
   \end{thebibliography}
\end{document}